\documentclass[12pt, reqno]{amsart}
\usepackage{amssymb,nag}
\usepackage{amsmath}
\usepackage[utf8]{inputenc}
\usepackage{longtable}
\usepackage{cite}
\usepackage[all,cmtip]{xy}
\usepackage{color}
\usepackage{fullpage}

\newcommand{\longhookrightarrow}{\ensuremath{\lhook\joinrel\relbar\joinrel\rightarrow}}

\usepackage{amsfonts}
\usepackage{amssymb}
\usepackage{amsthm}
\usepackage{mathtools}
\usepackage{mathrsfs}
\usepackage{float}

\theoremstyle{definition}
\newtheorem{defi}{Definition}[section]
\theoremstyle{plain}
\newtheorem{thm}[defi]{Theorem}

\newtheorem{prop}[defi]{Proposition}
\newtheorem{cor}[defi]{Corollary}
\newtheorem{lemma}[defi]{Lemma}
\theoremstyle{remark}

\newtheorem{rmk}[defi]{Remark}

\theoremstyle{definition}

\newcommand{\longra}{\longrightarrow}

\newcommand{\ka}{{\mathcal A}}

\newcommand{\km}{{\mathcal M}}

\newcommand{\kr}{{\mathcal R}}

\newcommand{\kv}{{\mathcal V}}

\newcommand{\kz}{{\mathcal Z}}

\newcommand{\IG}{{\mathbb G}}

\newcommand{\IP}{{\mathbb P}}

\newcommand{\Z}{{\mathbb Z}}

\newcommand{\rH}{{\rm H}}

\newcommand{\codim}{\operatorname{codim}}
\newcommand{\Pic}{\operatorname{Pic}}

\newcommand{\End}{\operatorname{End}}

\newcommand{\Sym}{\operatorname{Sym}}

\newcommand{\xdashrightarrow}[2][]{\ext@arrow 0359\rightarrowfill@@{#1}{#2}}
\newcommand{\xdashleftarrow}[2][]{\ext@arrow 3095\leftarrowfill@@{#1}{#2}}
\newcommand{\xdashleftrightarrow}[2][]{\ext@arrow 3359\leftrightarrowfill@@{#1}{#2}}
\def\rightarrowfill@@{\arrowfill@@\relax\relbar\rightarrow}
\def\leftarrowfill@@{\arrowfill@@\leftarrow\relbar\relax}
\def\leftrightarrowfill@@{\arrowfill@@\leftarrow\relbar\rightarrow}
\def\arrowfill@@#1#2#3#4{%
  $\m@th\thickmuskip0mu\medmuskip\thickmuskip\thinmuskip\thickmuskip
   \relax#4#1
   \xleaders\hbox{$#4#2$}\hfill
   #3$%
}

	\title{On isogenies of Prym varieties}
	\author{Roberto Laface and C\'esar Mart\'inez}
	\email{laface@math.uni-hannover.de}
    \email{cesar.martinez@unicaen.fr}

\begin{document}

\thispagestyle{empty}
\begin{abstract}
We prove an extension of the Babbage-Enriques-Petri theorem for semi-ca\-no\-ni\-cal curves. We apply this to show that the Prym variety of a generic element of a codimension $k$ subvariety of $\kr_g$ is not isogenous to another distinct Prym variety, under some mild assumption on $k$.
\end{abstract}
\maketitle

\setcounter{tocdepth}{1}

\section{Introduction}
Let $\mathcal{R}_g$ denote the moduli space of unramified irreducible double covers of complex smooth curves of genus~$g$.
Given an element $\pi:D\rightarrow C$ in $\mathcal{R}_g$, we can lift this morphism to the corresponding Jacobians via the norm map
\[
\mathrm{Nm}_{\pi}:J(D)\rightarrow J(C).
\]
By taking the neutral connected component of its kernel, we obtain an abelian variety of dimension $g-1$ called the \emph{Prym variety} attached to~$\pi$.

In this note, we study the isogeny locus in $\ka_{g-1}$ of Prym varieties attached to generic elements in $\kr_g$;
that is, principally polarized abelian varieties of dimension $g-1$ which are isogenous to such Prym varieties.
More concretely, given a subvariety $\kz$ of $\kr_g$ of codimension~$k$ and a generic element $\pi:D\rightarrow C$ in $\kz$,
we prove that the Prym variety attached to $\pi$ is not isogenous to a distinct Prym variety, whenever $g \geq \max \lbrace 7, 3k+5 \rbrace$,
see Theorem~\ref{thm1}.

This result is an extension of the analogue statements for Jacobians of generic curves proven by Bardelli and Pirola~\cite{bardelli-pirola89} for the case $k=0$, and Marcucci, Naranjo and Pirola~\cite{marcucci-naranjo-pirola16} for $k>0$, $g \geq 3k+5$ or $k=1$ and $g\geq 5$.
In the latter, to prove the case $g \geq 3k+5$, they use an argument on infinitesimal variation of Hodge structure proposed by Voisin in~\cite[Remark~(4.2.5)]{bardelli-pirola89} which allows them to translate the question to a geometric problem of intersection of quadrics.
In doing so, they give a generalization of Babbage-Enriques-Petri's theorem which allows them to recover a canonical curve from the intersection of a system of quadrics in $\mathbb{P}^{g-1}$ of codimension~$k$.
The strategy we follow to prove Theorem~\ref{thm1} is an adaptation of these techniques to the setting of Prym varieties.
We are also able to give an extension of Babbage-Enriques-Petri's theorem for semicanonical curves
in a similar fashion as in~\cite{marcucci-naranjo-pirola16}, see Proposition~\ref{prop1}.
Our result generalises the one by Lange and Sernesi~\cite{lange-sernesi96} for curves of genus $g\geq 9$,
since it recovers a semicanonical curve of genus $g\geq 7$ from a system of quadrics in $\mathbb{P}^{g-2}$ of codimension~$k$, $g \geq 3k+5$.

\medskip \noindent {\textbf{Acknowledgments.}} We warmly thank the organizing committee of Pragmatic for the generous funding and the opportunity of enjoying such a lively and stimulating environment.
We thank Juan Carlos Naranjo and V\'ictor González Alonso for proposing this problem and for their helpful advice.

\section{Intersection of quadrics}
Let $C$ be a smooth curve. Given a globally generated line bundle $L\in\Pic(C)$,
we denote by $\varphi_L:C\rightarrow \IP \mathrm{H}^0(C,L)^*$ its induced morphism.
If $L$ is very ample, we say that $\varphi_L(C)$ is \emph{projectively normal} if its homogeneous coordinate ring is integrally closed; or equivalently, if for all $k\geq 0$, the homomorphism
\[
\Sym^k \rH^0(L) \longra \rH^0(L^{\otimes k})
\]
is surjective.

We also recall that the \emph{Clifford index} of $C$ is defined as
\[
\min \lbrace \deg(L)-2\mathrm{h}^0(C,L)+2\rbrace,
\]
where the minimum ranges over the line bundles $L\in\Pic(C)$ such that $\mathrm{h}^0(C,L)\geq 2$ and~$\mathrm{h}^0(C,\omega_C\otimes L^{-1})\geq 2$. Its value is an integer between $0$ and~$\lfloor \frac{g-1}{2}\rfloor$, where $g$ is the genus of the curve.


Let $C$ be of genus $g$ and with Clifford index $c$. For any non-trivial $2$-torsion point $\eta$ in the Jacobian of~$C$, we call $\omega_C \otimes \eta$ a \emph{semicanonical line bundle} of $C$ whenever it is globally generated,
and we denote by $\varphi_{\omega_C \otimes \eta}:C\rightarrow\mathbb{P}^{g-2}$ its associated morphism.
In that case, we call its image $C_\eta:=\varphi_{\omega_C \otimes \eta}(C)$ a \emph{semicanonical curve}.
The following is a result of Lange and Sernesi~\cite{lange-sernesi96}, and Lazarsfeld~\cite{lazarsfeld89}:

\begin{lemma}\label{lem:proj-normal}
If $g\geq 7$ and $c\geq 3$, then $\omega_C\otimes\eta$ is very ample and the semicanonical curve $C_\eta$ is projectively normal.
\end{lemma}

Furthermore, Lange and Sernesi prove that $C_\eta$ is the only non-degenerate curve in the intersection of all quadrics in $\mathbb{P}^{g-2}$ containing $C_\eta$ if $c>3$, or $c=3$ and $g\geq 9$, see \cite{lange-sernesi96}.
The following proposition generalises this result for a smaller family of quadrics.

\begin{prop}\label{prop1}
Let $C$ be a curve of genus $g$ and Clifford index~$c$, and $\eta$ be a non-trivial $2$-torsion point in~$J(C)$.
Let $I_2(C_\eta) \subset \Sym^2 \rH^0(C,\omega_C \otimes \eta)$ be the vector space of equations of the quadrics containing~$C$,
and $K \subset I_2(C_\eta)$ be a linear subspace of codimension~$k$.
If $g \geq \max \lbrace 7, 2k+6\rbrace$ and $c \geq \max \lbrace 3, k+2 \rbrace$, then $C_\eta$ is the only irreducible non-degenerate curve in the intersection of the quadrics of~$K$.
\end{prop}

Notice that for $k=0$, this proposition extends the result of Lange and Sernesi \cite{lange-sernesi96} to the cases when $c=3$ and $g=7$ and~$8$. We refer to Remark~\ref{rmk1} for a brief discussion on a simplified version of the following proof in this case.

\begin{proof}
We start by assuming that there exists an irreducible non-degenerate curve ${C}_0$ in the intersection of quadrics
$\bigcap_{Q\in K}Q\subset \mathbb{P}\mathrm{H}^0(C,\omega_C\otimes \eta)^*$,
which is different from~$C_\eta$.
In particular, we can choose $k+1$ linearly independent points in $\bigcap_{Q\in K}Q$ such that $x_i\not\in C_\eta$ for all~$i$.
By abuse of notation, we denote also as $x_i$ the representatives in $\mathrm{H}^0(C,\omega_C\otimes\eta)^*$.
We define $L\subset \Sym^2\mathrm{H}^0(C,\omega_C\otimes\eta)^*$ as the linear subspace spanned by $x_i\otimes x_i$. 

Let $R=I_2(C_\eta)/K$ and $R'=\Sym^2\mathrm{H}^0(C,\omega_C\otimes\eta)/K$.
By Lemma~\ref{lem:proj-normal} and the fact that $g\geq 7$ and $c\geq 3$, we have that $C_\eta$ is projectively normal.
Hence, we can build the following diagram:
\[
\xymatrix{
 & & 0 \ar[d] & 0 \ar[d] & & \\
 & 0 \ar[r] & K \ar[d] \ar[r] & I_2(C_\eta) \ar[d] \ar[r] & R \ar[r] & 0\\
 & & \Sym^2\mathrm{H}^0(C,\omega_C\otimes\eta) \ar[d] \ar@{=}[r] & \Sym^2\mathrm{H}^0(C,\omega_C\otimes\eta) \ar[d] & & \\
 0 \ar[r] & R \ar[r] & R' \ar[d] \ar[r] & \mathrm{H}^0(C,\omega_C^{\otimes 2}) \ar[d] \ar[r] & 0 & \\
 & & 0 & 0 & &
}
\]
where the last row is obtained by applying the snake lemma to the first two rows. By dualizing this diagram, we get
\[
\xymatrix{
 & & 0 \ar[d] & 0 \ar[d] & & \\
& 0 \ar[r] & \mathrm{H}^0(C,\omega_C^{\otimes 2})^*=\rH^1(C,T_C) \ar[d] \ar[r] & R'^* \ar[d] \ar[r] & R^* \ar[r] & 0\\
 & & \Sym^2\mathrm{H}^0(C,\omega_C\otimes\eta)^* \ar[d] \ar@{=}[r] & \Sym^2\mathrm{H}^0(C,\omega_C\otimes\eta)^* \ar[d] & & \\
 0 \ar[r] & R^* \ar[r] & I_2(C_\eta)^* \ar[d] \ar[r] & K^* \ar[d] \ar[r] & 0 & \\ 
 & & 0 & 0 & &
}
\]

Notice that $Q(\alpha)=0$ for every $\alpha\in L$ and every $Q\in K$.
Therefore, $L\subset R'^*$
Since $\dim(L)=k+1$ and $\dim(R)=k$, there is a non-trivial element $\alpha\in L\cap \mathrm{H}^1(C,T_C)$.
By the isomorphism $\mathrm{H}^1(C,T_C)\simeq\mathrm{Ext}^1(\omega_C,\mathcal{O}_C)$,
there is a $2$ vector bundle $E_\alpha$ associated to $\alpha$ satisfying the following exact sequence:
\begin{equation}\label{eq:1}
0\longrightarrow\mathcal{O}_C\longrightarrow E_\alpha\longrightarrow\omega_C\longrightarrow 0.
\end{equation}
The cup product with $\alpha$ is the coboundary map $\mathrm{H}^0(C,\omega_C)\rightarrow\mathrm{H}^1(C,\mathcal{O}_C)$.
By writing the element $\alpha=\sum_{i=1}^{k+1}a_ix_i\otimes x_i$, we have
\[
\mathrm{Ker}(\cdot\cup\alpha)=\bigcap_{i\;\mid\;a_i\neq 0}H_i,
\]
where $H_i=\mathrm{Ker}(x_i)$.
After reordering, we may assume that $x_1,\ldots,x_{k'}$ are the points such that $a_i\neq 0$, for some $k'\leq k+1$.
This means that there are $g-k'$ linearly independent sections in $\mathrm{H}^0(C,\omega_C)$ lifting to~$\mathrm{H}^0(C,E_\alpha)$.
Denote by $W\subset\mathrm{H}^0(C,E_\alpha)$ the vector space generated by these sections, and consider the morphism
$\psi:\wedge^2W\rightarrow\mathrm{H}^0(C,\omega_C)$ obtained by the following composition:
\[
\wedge^2W\longhookrightarrow \wedge^2 \mathrm{H}^0(C,E_\alpha)
\longrightarrow\mathrm{H}^0(C,\det E_\alpha)=\mathrm{H}^0(C,\omega_C).
\]
The kernel of $\psi$ has codimension at most~$g$, and the Grassmannian of the decomposable elements in $\mathbb{P}(\wedge^2W)$ has dimension $2(g-k'-2)$.
Since $g>2k+5$ by hypothesis, we have that their intersection is not trivial.
Thus, take $s_1,s_2\in \mathrm{H}^0(C,E_\alpha)$ such that $\psi(s_1\wedge s_2)=0$.
They generate a line bundle $M_\alpha\subset E_\alpha$ and $h^0(C,M_\alpha)\geq 2$.
Take $Q_\alpha$ the neutral component of the quotient $E_\alpha/M_\alpha$,
and $L_\alpha$ the kernel of $E_\alpha\rightarrow Q_\alpha$, then we obtain the following exact sequence:
\begin{equation}\label{eq:2}
0\longrightarrow L_\alpha\longrightarrow E_\alpha\longrightarrow Q_\alpha\longrightarrow 0.
\end{equation}
Notice that $M_\alpha\subset L_\alpha$, hence $h^0(C,L_\alpha)\geq 2$.
Moreover from \eqref{eq:1} and \eqref{eq:2}, we obtain $\omega_C\simeq\det E_\alpha\simeq L_\alpha\otimes Q_\alpha$, which implies that $Q_\alpha\simeq \omega_C\otimes L_\alpha^{-1}$. We have the following diagram:
\[
\xymatrix{
 & & 0 \ar[d] & & \\
 & & L_\alpha \ar[d] & & \\
 0 \ar[r] & \mathcal{O}_C \ar[r] \ar[rd] & E_\alpha \ar[d] \ar[r] & \omega_C \ar[r] & 0\\
 & & \omega_C\otimes L_\alpha^{-1} \ar[d] & & \\
 & & 0 & &
}
\]

Assume that $\mathcal{O}_C\rightarrow\omega_C\otimes L_\alpha^{-1}$ is~$0$. Then the section of $E_\alpha$ that represents $\mathcal{O}_C\rightarrow E_\alpha$ would be a section of $L_\alpha$, in particular, a section in~$W$.
Since the sections in $W$ map to sections of $\omega_C$, this contradicts the exactness of the horizontal sequence.
So $\mathcal{O}_C\rightarrow\omega_C\otimes L_\alpha^{-1}$ is not~$0$ and the $h^0(C,\omega_C\otimes L_\alpha^{-1})>0$.

If $h^0(C,\omega_C\otimes L_\alpha ^{-1})\geq 2$, we have that
\begin{equation}\label{eq:3}
c\leq \deg (L_\alpha)-2h^0(C,L_\alpha)+2.
\end{equation}
Moreover, $h^0(C,L_\alpha)+h^0(C,\omega_C\otimes L_\alpha^{-1})\geq h^0(C,E_\alpha)> \dim (W)=g-k'$ and, using Riemann-Roch we obtain that $2h^0(C,L_\alpha)\geq \deg (L_\alpha)+2-k'$.
Combining this with \eqref{eq:3}, we obtain that $c\leq k'\leq k+1$ which contradicts our hypothesis on~$c$ ($c\geq k+2$).
Hence, $h^0(C,\omega_C\otimes L^{-1}_\alpha)=1$.

Write $\omega_C\otimes L^{-1}_\alpha\simeq \mathcal{O}_C(p_1+\cdots+p_e)$, where $e=\deg(\omega_C\otimes L^{-1}_\alpha)$.
Notice that $h^0(C,L_\alpha)\geq g-k'$ and $\deg(L_\alpha)=2g-2-e$.
Using Riemann-Roch, we get
\[
g-k'\leq h^0(C,L_\alpha)=h^0(C,\omega_C\otimes L^{-1}_\alpha)+2g-2-e-(g-1)=g-e.
\]
So $e\leq k'$.

By \eqref{eq:2}, we have that $L_\alpha\simeq \omega_C(-p_1-\cdots -p_e)$.
Moreover, the sections of $L_\alpha$ lie in $W$, and by construction of $W$ we have that $\mathrm{H}^0(\omega_C(-p_1-\cdots-p_e))\subset\mathrm{Ker}(\cdot\cup\alpha)=\cap_{i\mid a_i\neq 0}H_i$.
Therefore, by dualizing this inclusion, we obtain that
\begin{equation}\label{eq:inclusion}
\langle x_1,\ldots,x_{k'}\rangle_{\mathbb{C}}
\subset
\langle p_1,\ldots,p_e\rangle_{\mathbb{C}}.
\end{equation}

Let $\gamma:N_0\rightarrow {C}_0$ be a normalization.
For any generic choice of $k+1$ points $x_i\in N_0$, we can repeat the construction above for $\gamma(x_1),\ldots,\gamma(x_{k+1})$, and we can assume that $k'$ and $e$ are constant.
We define the correspondence
\[
\Gamma=\left\lbrace (x_1+\cdots+x_{k'},p_1+\cdots+p_e)\in N_0^{(k')}\times C_\eta^{(e)}\;\mid\;
\langle\gamma(x_1),\ldots,\gamma(x_{k'})\rangle_{\mathbb{C}}\subset\langle p_1,\ldots,p_e\rangle_{\mathbb{C}}\right\rbrace.
\]
Observe that $\Gamma$ dominates $N_0^{(k')}$, so $e\leq k'\leq \dim\Gamma$.
In addition, the second projection $\Gamma\rightarrow C_\eta^{(e)}$ has finite fibers, since both curves are non-degenerate. This implies that $\dim\Gamma\leq e$, and so we have $k'=e$.
Since $k'\leq k+1\leq g-3$, by the uniform position theorem we have that the rational maps
\begin{align*}
C^{(k')}\dashrightarrow Sec^{(k')}(C_\eta)\subset \IG(e-1,\mathbb{P}^{g-2}),\\
N_0^{(k')}\dashrightarrow Sec^{(k')}(N_0)\subset \IG(e-1,\mathbb{P}^{g-2}),
\end{align*}
are generically injective.
This gives a birational map between $C_\eta^{(k')}$ and $N_0^{(k')}$. In particular, it induces dominant morphisms
$JC_\eta\rightarrow JN_0$ and $JN_0\rightarrow JC_\eta$. Therefore, $g(C_\eta)=g(N_0)$ and by a
theorem of Ran~\cite{ran86}, the birational map $C_\eta^{(k')}\dashrightarrow N_0^{(k')}$ is defined by a birational map between $C_\eta$ and~$N_0$.
By composing it with the normalization map $\gamma$, we obtain a birational map
\[
\varphi: C_\eta\dashrightarrow C_0,
\]
that defines the correspondence $\Gamma$; that is $\langle \varphi(x_1),\ldots,\varphi(x_{k'})\rangle=
\langle x_1,\ldots,x_{k'}\rangle$ for generic elements $x_1+\ldots +x_{k'}\in C_\eta^{(k')}$.
This implies that $\varphi$ is generically the identity map over~$C_\eta$.
Thus $C_\eta=C_0$, which is a contradiction and ends the proof.
\end{proof}

\begin{rmk}\label{rmk1}
The proof of Corollary~\ref{cor1} can be simplified for the case $K=I_2(C_\eta)$, that is~$k=0$.
Under this assumption, we only consider one point $x\not\in C_\eta$, and $k'=e=1$.
Therefore, the inclusion~\eqref{eq:inclusion} already implies the equality $C_\eta=C_0$.
\end{rmk}

\section{Main theorem}
An element in $\mathcal{R}_g$ can be identified with a pair $(C,\eta)$, where $C$ is a complex smooth curve of genus~$g$,
and $\eta$ is a non-trivial $2$-torsion element in the Jacobian of~$C$.
This allows us to consider $\mathcal{R}_g$ as a covering of the moduli space $\mathcal{M}_g$ of complex smooth curves of genus~$g$.
It is given by the morphism
\[
\mathcal{R}_g\longrightarrow\mathcal{M}_g,\quad (C,\eta)\longmapsto C,
\]
which has degree $2^{2g}-1$.
Thus, a generic choice of an element in a subvariety $\mathcal{Z}\subset\mathcal{R}_g$ is equivalent to a generic choice of a curve $C$ in
the image of $\mathcal{Z}$ in $\mathcal{M}_g$, and any non-trivial element~$\eta\in J(C)[2]$.

The following result is a direct consequence
of Proposition~\ref{prop1} and it is the version of Babbage-Enriques-Petri's theorem that we use in the proof of the main result in this article.

\begin{cor}\label{cor1}
Let $(C,\eta)$ be a generic point in a subvariety $\kz$ of $\kr_g$ of codimension~$k$.
Let $I_2(C_\eta) \subset \Sym^2 \rH^0 (C,\omega_C \otimes \eta)$ be the vector space of the equations of quadrics in $\mathbb{P}^{g-2}$ containing~$C_\eta$. Let $K \subset I_2(C_\eta)$ be a linear subspace of codimension~$k$.
If $g \geq \max\lbrace 7,3k+5\rbrace$, then $C_\eta$ is the only irreducible non-degenerate curve in the intersection of the quadrics of~$K$. 
\end{cor}

\begin{proof}
Let $\mathcal{M}_g^c$ be the locus in $\mathcal{M}_g$ corresponding to curves with Clifford index~$c$.
Then $\km_g^c$ is a finite union of subvarieties of~$\km_g$,
where the one of higher dimension corresponds to the curves whose Clifford index is realized by a~$g^1_{c+2}$ linear series, see~\cite{coppens-martens91}.
By Riemann-Hurwitz, the codimension in $\mathcal{M}_g$ of the component of the curves with a $g^1_{c+2}$ linear series is
\[
3g-3 - (2g-2c+2-3) = g-2c-2.
\]
If $k=0$, a generic curve in $\mathcal{M}_g$ has Clifford index $c\geq 3$, because $g\geq 7$.
As when~$k>0$, since~$g \geq 3k+5$, we obtain
\[
k \geq g -2c -2 \geq 3k +5 -2c-2 = 3k -2c+3,
\]
and thus $c \geq k+2$.
The corollary follows by applying Proposition~\ref{prop1}.
\end{proof}

Let $\widetilde{\ka_g}^m$ be the space of isogenies of principally polarized Abelian varieties of degree $m$ (up to isomorphism); that is the space of classes of isogenies $\chi: A \longra A'$ such that $\chi^*L_{A'} \cong L_{A}^{\otimes m}$,
where~$L_A$ (respectively~$L_{A'}$) is a principal polarization on $A$ (respectively~$A'$). There are two forgetful maps to the moduli space $\mathcal{A}_g$ of p.p.a.v. of dimension~$g$
\begin{equation}\label{eq:forgetful}
\xymatrix{
 & \hspace*{3mm}\widetilde{\ka_g}^m \ar[dl]_{\varphi} \ar[dr]^{\psi}& \\
	\ka_g  & &   \ka_g,
}
\end{equation}
such that $\varphi(\chi) = (A,L_{A})$ and $\psi(\chi) = (A',L_{A'})$. These maps yield the following commutative diagram,
\begin{equation}\label{tangent}
\xymatrix{
	& \hspace*{3mm}T_{[\chi]}\widetilde{\ka_{g-1}}^m \ar[dl]_{d\varphi} \ar[dr]^{d\psi} & 
    \vspace*{2mm}\\
	\ T_{[A]} \ka_{g-1} \ar[rr]^\lambda_{} & & T_{[A']} \ka_{g-1}.
}
\end{equation}
where all maps are isomorphisms.

\begin{thm}\label{thm1}
Let $\kz \subset \kr_g$ be a (possibly reducible) codimension $k$ subvariety. Assume that $g \geq \max\lbrace 7,3k+5 \rbrace$, and let $(C,\eta)$ be a generic element in~$\kz$.
If there is a pair $(C',\eta')\in\mathcal{R}_g$ such that there exists an isogeny $\chi: P(C,\eta) \longra P(C',\eta')$,
then $(C,\eta) \cong (C',\eta')$ and $\chi = [n]$, for some $n \in \Z$. 
\end{thm}

\begin{proof}
Suppose that $(C,\eta)$ is generic in $\kz$. By the assumption on $g$, the Clifford index of a generic element of $\kz$ is at least three (as shown in the proof of Corollary~\ref{cor1}). However, by \cite{naranjo96}, if the Clifford index of a curve $C$ is $c \geq 3$, then the corresponding fiber of the Prym map is 0-dimensional, i.e.~$\dim P^{-1}(P(C,\eta)) = 0$.
Therefore, the restriction of the Prym map to $\kz$,
\[
P_{\vert \kz}: \kz \longhookrightarrow \kr_g \longra \ka_{g-1},
\]
has generically fixed degree $d$ onto its image, for some $d\in\mathbb{N}$.
So, by the genericity of the pair~$(C,\eta)$, we can assume that $(C,\eta)$ lies in the locus of $\kz$ where~$P_{\vert\kz}$ is étale.
This gives the isomorphisms of the tangent spaces
\begin{equation}\label{eq:tangent}
T_{P[(C,\eta)]}P(\kz) \cong T_{[C,\eta]}\kz\quad \mbox{and}\quad T_{P[(C,\eta)]}P(\kr_g) \cong T_{[C,\eta]}\kr_g.
\end{equation}

Let us assume that the locus of curves in $\kr_g$ whose corresponding Prym variety is isogenous to the Prym variety of an element in $\kz$ has an irreducible component $\kz'$ of codimension $k$. By \cite{pirola88}, since $k < g-2$, we have $\End(P(C,\eta)) \cong \Z$. Suppose that we are given an isogeny $\chi: P(C,\eta) \longra P(C',\eta')$; then, it must have the property that the pull-back of the principal polarization $\Xi'$ is a multiple of the principal polarization $\Xi$ on $P(C,\eta)$, say $\chi^*\Xi' \cong \Xi^{\otimes m}$, for some $m\in\mathbb{Z}$. 

For such $m$, we have the diagram of forgetful maps as in~\eqref{eq:forgetful} with $g-1$ in place of~$g$.
We can find an irreducible subvariety $\kv \subset \widetilde{\ka_{g-1}}^m$ which dominates both $P(\kz)$ and $P(\kz')$ through $\varphi$ and $\psi$ respectively.
Setting $\kr := \varphi^{-1}(P(\kr_g))$ and $\kr' := \psi^{-1}(P(\kr_g))$, we have the inclusion~$\kv \subset \kr \cap \kr'$. 

For a generic element $\chi: P(C,\eta) \longra P(C', \eta')$ in $\kv$, the diagram (\ref{tangent}) becomes
\[
\xymatrix{
& \hspace*{3mm}T_{[\chi]}\widetilde{\ka_{g-1}}^m \ar[dl]_{d\varphi} \ar[dr]^{d\psi} & \vspace*{2mm}\\
\ T_{[P(C,\eta)]} \ka_{g-1} \ar[rr]^\lambda_{\cong} & & T_{[P(C',\eta')]} \ka_{g-1}.
}
\]
In addition, $T_{[P(C,\eta)]} \ka_{g-1} \cong \Sym^2 \rH^0(P(C,\eta),T_{P(C,\eta)}) \cong \Sym^2 \rH^0(\omega_C \otimes \eta)^*$. By looking at~$d\varphi$, and the isomorphisms in~\eqref{eq:tangent}, we see that we have the following diagram of tangents spaces and identifications:
\[
\xymatrix{
T_{[\chi]}\kv \ar[d]_{\cong} \ar@{^{(}->}[rr] & &T_{[\chi]}\kr \ar[d]_{\cong}\ar@{^{(}->}[rr] & &T_{[\chi]}\kr +T_{[\chi]}\kr'\ar@{=}[d]\ar@{^{(}->}[rr] & &T_{[\chi]}\widetilde{\ka_{g-1}}^m\ar[d]_{\cong}\\
T_{[C,\eta]}\kz \ar@{^{(}->}[rr] & &T_{C,\eta]}\kr_g \ar@{^{(}->}[rr] & &\bar{T}\ar@{^{(}->}[rr] & &\Sym^2 \rH^0(\omega_C \otimes \eta)^*
}
\]
where the vertical arrows are $d\varphi$. 

By the Grassmann formula, $\dim \bar{T} \leq 3g - 3 +k$. Set 
\[K(C_\eta) := \ker \Big(\Sym^2 \rH^0(\omega_C \otimes \eta) \longra \bar{T}^*\Big),\]
which is a subspace of the space of quadrics containing the semicanonical curve~$C_\eta$.
Notice that $\codim_{I_2(C_\eta)} K(C_\eta) \leq k$. By repeating the above argument with $\psi$ in place of $\varphi$, we get the corresponding inclusion of vector spaces $K(C'_{\eta'}) \subset I_2(C'_{\eta'})$, and by using the (canonical) isomorphism $\lambda$ above, we get a (canonical) isomorphism $K(C_\eta) \cong K(C'_{\eta'})$. 

A closer look at $\lambda: T_{[P(C,\eta)]} \ka_{g-1} \longra T_{[P(C',\eta')]}\ka_{g-1}$ reveals that this map is induced by the isogeny $\chi : P(C,\eta) \longra P(C',\eta')$.
In fact, one has that $d_0 \chi: \rH^0(\omega_{C} \otimes \eta) \longra \rH^0(\omega_{C'} \otimes \eta')$ is an isomorphism, and $\lambda$ is induced by it. This means that $d_0 \chi$ induces an isomorphism of projective spaces $\IP\rH^0(\omega_{C} \otimes \eta)^* \longra \IP\rH^0(\omega_{C'} \otimes \eta')^*$, which sends quadrics containing $C'_{\eta'}$ to quadrics containing $C_\eta$, by means of $\lambda$.
By using Lemma~\ref{cor1}, we get that $C_\eta \cong C'_{\eta'}$, and thus $C \cong C'$. This gives us the following commutative diagram
\[
\xymatrixcolsep{5pc}\xymatrix{
C \ar[d]^\cong \ar@{^{(}->}[r]^{\varphi_{\omega_C \otimes \eta}}&C_\eta\ar@{^{(}->}[r]\ar[d]^\cong & \IP\rH^0(\omega_{C} \otimes \eta)^* \ar[d]^\cong\\
C' \ar@{^{(}->}[r]^{\varphi_{\omega_{C'} \otimes \eta'}}&C'_{\eta'}\ar@{^{(}->}[r] & \IP\rH^0(\omega_{C'} \otimes \eta')^*
}
\]
from which we deduce that $(C,\eta) \cong (C',\eta')$.
Indeed, pulling back hyperplanes to $C$ and~$C'$, yields an isomorphism $\omega_{C'} \otimes \eta' \cong \omega_{C} \otimes \eta$, from which it follows that~$\eta \cong \eta'$.
The isogeny is necessarily of the form $[n]$, for some $n\in\Z$, because $\End(P(C,\eta))\cong \Z$.
\end{proof}

\bibliographystyle{plain}

\end{document}